\documentclass[11pt]{amsart}
\usepackage[hmargin=3cm,vmargin=3cm]{geometry}

\usepackage{latexsym,amsfonts,amssymb}
\usepackage{amsmath,amsthm,amscd}
\def\al{\alpha}

\def\la{\lambda}
\def\La{\Lambda}

\usepackage{bbm}

\def\Z{{\mathbbm Z}}
\def\Q{{\mathbbm Q}}
\def\R{{\mathbbm R}}
\def\C{{\mathbbm C}}
\def\co{\colon}
\newcommand{\ringone}{\Bbbk}
\newcommand{\ringtwo}{k}
\newcommand{\tLa}{\widetilde{\La}}

\newcommand{\tR}{\widetilde{R}}
\DeclareMathOperator{\Hom}{Hom}
\DeclareMathOperator{\aff}{aff}
\DeclareMathOperator{\Span}{Span}

\theoremstyle{definition}
\newtheorem{thm}{Theorem}[section]
\newtheorem{defn}{Definition}[section]
\newtheorem{remark}{Remark}[section]
\newtheorem{lem}{Lemma}[section]
\newtheorem{cor}{Corollary}[section]
\newtheorem{ex}{Example}[section]

\title[]{Coxeter groups preserving orthants}

\author[]{Ben Elias}
\address{University of Oregon.}
\email{belias@uoregon.edu}%\thanks{Blah blah}

\begin{document}
	
\begin{abstract}
In this short, elementary note we prove that if a faithful reflection representation of a Coxeter group preserves an orthant, then that Coxeter group is a product of symmetric
groups acting on its natural permutation representation. We also prove an affine analogue of this statement, where an orthant is preserved modulo an invariant sublattice. As a
consequence, the existence of two different versions of the quantum geometric Satake equivalence is a purely type $A$ phenomenon. \end{abstract}

\maketitle

\section{Introduction}

Let $T$ be the diagonal torus of $GL_n(\C)$, and let $\{x_1, \ldots, x_n\}$ be a basis of its character lattice $\La = \Z^n$. The Weyl group $W = S_n$ acts on $\La$ by
permuting this basis, a seemingly innocent point which has some interesting consequences. An \emph{orthant} is the set of non-negative linear combinations of a basis; it is the generalization of a quadrant beyond two dimensions. Obviously, if a group action on a permutes a basis of a $\Z$-module, then it preserves an orthant.

Let $R_H$ denote the polynomial ring $\Z[x_1, \ldots, x_n]$. An element $\la \in \La$ can be interpreted as a linear polynomial in $R_H$. The ring $R_H$ is the $T$-equivariant
cohomology of a point, and inherits a natural $W$ action. Meanwhile, let $R_K$ denote the group algebra $\Z[\La]$, which we think of as being spanned by elements $e^{\la}$ for $\la
\in \La$. The ring $R_K$ is the $T$-equivariant $K$-theory of a point. The natural action of $W$ on $R_K$ is $w(e^{\la}) = e^{w(\la)}$. Moreover, $R_K$ is isomorphic to a Laurent
polynomial ring $\Z[y_1^{\pm 1}, \ldots, y_n^{\pm 1}]$, where $y_i = e^{x_i}$.

Let $p = x_1 \cdots x_n$. What is remarkable here is that the map $\epsilon \co x_i \mapsto e^{x_i} = y_i$ extends to a $W$-equivariant ring homomorphism $R_H \to R_K$, which makes $R_K$ isomorphic to the
localization $R_H[p^{-1}]$. Said another way, the lattice $\La$ contains a $W$-invariant orthant, and consequently the ring $R_K$ contains a $W$-invariant (non-Laurent) polynomial subring $\Z[y_1, \ldots, y_n]$ (which is isomorphic to $R_H$).

\begin{remark} It is not reasonable to call $\epsilon$ an ``exponentiation'' map, since it does not send $\la$ to $e^{\la}$ in general; for example, the
simple root $x_1 - x_2$ is sent to $e^{x_1} - e^{x_2}$, which does not agree with $e^{x_1 - x_2}$. \end{remark}

Let us briefly motivate why we care about the ring isomorphism $R_K \cong R_H[p^{-1}]$. Using the ring $R_H$ one can construct the category of Soergel bimodules \cite{Soer07}, which
form a full monoidal subcategory of graded $(R_H, R_H)$-bimodules. Soergel bimodules categorify the Hecke algebra and play a major role in geometric representation theory. More
recently, the author and Williamson have introduced $K$-theoretic Soergel bimodules \cite{EWKSBim}, see also \cite{Eberhardt}. In \cite{EWKSBim} we prove that, for $GL_n$, the
category of $K$-theoretic Soergel bimodules is just the localization of the category of ordinary Soergel bimodules which inverts $p$. Since none of the roots $x_i - x_j$ divide $p$,
this localization is fairly harmless, see \cite[\S 5.1]{EWKSBim}.

However, this phenomenon does not even extend from $GL_n(\C)$ to $SL_n(\C)$. There is no $W$-invariant orthant in the weight lattice of $SL_n$. For example, letting $\varpi$ denote
the highest weight of the standard representation, it is easy to verify that the $\Z_{\ge 0}$-span of the $W$-orbit of $\varpi$ is the entire weight lattice.

We could not find any discussion of $W$-invariant orthants in the literature. The first goal of this note is to point out how special $GL_n(\C)$ is in this
regard. Preserving an orthant is a priori a weaker condition than permuting a basis, but we prove that if $W$ is a Coxeter group acting by reflections on a lattice $\La$, and $\La$
has a $W$-invariant orthant, then $W$ permutes a basis, and the simple reflections act as transpositions. That is, up to isomorphism, the action of $W$ factors through the usual
action of $S_n$ on $\Z^n$.

\begin{remark} \label{rmk:nonfaithful} As a consequence, if $W$ acts faithfully then it is a parabolic subgroup of $S_n$; that is, $W$ is a product of type $A$ Weyl groups, and $\La$ is the direct sum of the corresponding permutation representations. If $W$ does not act faithfully there are many more options. To give one example, $S_3$ is a quotient of the Dihedral group of type $I_2(3k)$ for any $k \ge 1$. To give another example, $S_4$ is a quotient of the affine Weyl group of type $\tilde{A}_2$, where the three simple reflections are sent to the transpositions $(12)$, $(13)$, and $(14)$. \end{remark}

The second goal concerns a similar phenomenon in affine type. 

Let $\tLa$ be the weight lattice for the affine lie algebra $\widetilde{\mathfrak{gl}}_n$. It has basis $\{x_1, \ldots, x_n, \hbar\}$. The affine type $A$ Weyl group $W_{\aff}$ acts
on $\tLa$ as follows. For $1 \le i \le n-1$, the simple reflection $s_i$ sends $x_i \mapsto x_{i+1}$ as usual, whereas the affine reflection $s_0$ sends $x_n$ to $x_1 -\hbar$. The
element $\hbar$ is $W_{\aff}$-invariant. We call this the \emph{affine permutation representation} of $W_{\aff}$. Note that $\hbar$ is the sum of the simple roots, which are $\al_i
= x_i - x_{i+1}$ for $1 \le i \le n-1$, and $\al_0 = x_n - x_1 + \hbar$. Now we can define a polynomial ring $\tR_H := \Z[x_1, \ldots, x_n, \hbar]$ and a group algebra $\tR_K :=
\Z[\tLa]$.

There is no $W_{\aff}$-invariant orthant inside $\tLa$. However, there is the next best thing, an invariant ``almost-orthant:'' the set of elements of the form
\begin{equation} \sum a_i x_i + c \hbar, \qquad a_i \in \Z_{\ge 0}, c \in \Z, \end{equation} is preserved by $W_{\aff}$. Let $q = e^{-\hbar} \in \tR_K$. Since $q$ is
$W_{\aff}$-invariant, we treat it as an invertible scalar. Let $y_i := e^{x_i}$ as before. Then $\tR_K$ admits a polynomial subring $\tR_q := \Z[q,q^{-1}][y_1, \ldots, y_n]$, where
we think of it as a polynomial ring only because we have changed the base ring from $\Z$ to $\Z[q,q^{-1}]$. We have $s_i(y_i) = y_{i+1}$ for $1 \le i \le n-1$, and $s_0(y_n) = q y_1$. This time $\tR_q$ is not isomorphic to $\tR_H$.

Soergel bimodules for $\tR_H$ play a role in the geometric Satake equivalence. Soergel bimodules for $\tR_q$ were first studied in \cite{EQuantumI}, where they were shown to
participate in a quantum analogue of the geometric Satake equivalence. At the time, this $q$-deformation was mysterious, and it was not clear how to generalize it beyond type $A$,
or if such a generalization were possible. The nilHecke algebra associated to $\tR_q$ has also been studied, and unexpected features arise when $q$ is specialized to a root of unity, see \cite{EJY1}.

More recently, Soergel bimodules for $\tR_K$ were studied in \cite{EWKSBim}, where they are conjectured to participate in an (a priori different) quantum analogue of the geometric
Satake equivalence (which generalizes to any type). It was observed there that $\tR_K \cong \tR_q[p^{-1}]$ for $p=y_1 \cdots y_n$, and thus Soergel bimodules for $\tR_K$ are a
harmless localization of Soergel bimodules for $\tR_q$ (but not for $\tR_H$). Thus the mysterious quantum Satake of \cite{EQuantumI} actually arises as a polynomial
piece inside the $K$-theoretic Satake of \cite{EWKSBim}.

We also prove in this note that almost-orthants are special. Let $W$ be a Coxeter group acting on a lattice $\La$ by reflections. If $\{x_1, \ldots, x_n, \hbar_1, \ldots, \hbar_m\}$
is a basis of $\La$ for which $\hbar_j$ is fixed by $W$ for all $j$, then we call the subset $\Z_{\ge 0} \cdot \{x_i\} + \Z \cdot \{\hbar_j\}$ an \emph{almost-orthant}. If $\La$ has
an invariant almost-orthant then $W$ permutes the set $\{x_i\}$
modulo the span of $\{\hbar_j\}$, with simple reflections acting as transpositions. Moreover, if $W$ acts faithfully then $W$ is a product of finite and affine type $A$ Weyl groups
acting on the direct sum of their (affine) permutation representations, up to some possible modifications to the invariant part of the lattice, see Theorem \ref{thmwithmods} for details.

\begin{remark} For an example of such a modification, one could take the representation of $\tilde{A}_n \times \tilde{A}_m$ which is a direct sum of their affine permutation representations, and then identify the sum of the simple roots for $\tilde{A}_n$ with the sum of the simple roots for $\tilde{A}_m$. \end{remark}

In conclusion, outside of type $A$, there is no polynomial piece inside $K$-theoretic Satake, i.e. there is no analogue of $\tR_q$, and the approach of \cite{EQuantumI} should not
generalize.

\begin{remark} For the root lattice of $SL_n$ one can define analogues of both $\tR_q$ and $\tR_K$. The former is not a polynomial subring of the other. However, both are subrings of $\tR_K(GL_n)$, and their categories of Soergel bimodules can be related by way of Soergel bimodules for affine $GL_n$. Technically, it is the $SL_n$ analogue of $\tR_q$ which was studied in \cite{EQuantumI}.  \end{remark}
	
\noindent {\bf Acknowledgments.} We would like to thank Rachel Webb for very interesting conversations. The author was supported by NSF grant DMS-2201387, and appreciates the support to his research group from DMS-2039316.

\section{Reflection actions that preserve orthants}

% When $S_n$ acts on weight lattice of $\mathfrak{gl}_n$, it preserves an orthant. That is, the $\N$-span of some $\Z$-basis (the standard basis) is preserved. The same can be said for
% \begin{equation} \label{onlythis} S_{n_1} \times \cdots \times S_{n_k} \subset S_{n_1 + \ldots + n_k} \qquad \textrm{acting on} \qquad \Z^{n_1 + \ldots + n_k}.\end{equation} An even
% stronger statement can be made: this action permutes a basis of the orthant.
%
% The main theorem of this chapter is two converse statements. The first is that any action of a Weyl group which preserves an orthant actually permutes a basis. The second is that any action of a Weyl group which permutes a basis actually a case of the action in \eqref{onlythis}.
%
% We also prove some slight generalizations of this result, over other base rings. For example, an action of a Weyl group on a $\Q$-vector space which preserves the $\Q_{\ge 0}$-span of
% some basis might not permute the basis, but will permute it up to rescaling. An action of a Weyl group which permutes a basis up to rescaling is, up to isomorphism, actually a case of the action in \eqref{onlythis}, but with a rescaled basis.
%
% The proofs are quite straightforward, but we have never seen this result in the literature before.

%%%%%%%%%%%%%%%%%%%%%%
\subsection{Setup} \label{subsec:setup}
%%%%%%%%%%%%%%%%%%%%%%

A realization is one way to formulate the concept of a linear action of $W$ where $S$ acts by reflections.

\begin{defn} Let $\ringone$ be a commutative domain. Let $V$ be a free $\ringone$-module, and
$V^* := \Hom_{\ringone}(V,\ringone)$ be the dual module. Let $(W,S)$ be a Coxeter system. A \emph{realization} of $(W,S)$ over $\ringone$ is the data of $V$ and, for each $s \in S$, of a \emph{simple
root} $\al_s \in V$ and a \emph{simple coroot} $\al_s^{\vee} \in V^*$. These are requred to satisfy \begin{enumerate} \item $\al_s^\vee(\al_s) = 2$,
\item for $s \ne t$, $\al_s^\vee(\al_t) = 0$ if and only if $m_{st} = 2$, \item the assignment $s(x) := x - \al_s^\vee(x) \cdot \al_s$ defines an $\ringone$-linear action of $W$ on $V$. \end{enumerate}
\end{defn}

Any root datum gives rise to a realization in the obvious way, where one typically chooses $\ringone = \Q$ or $\ringone = \Z$. A realization is \emph{faithful} if the action of $W$ on $V$ is faithful.

\begin{ex} The \emph{permutation realization} of $S_n$ is the free $\ringone$-module with basis $\{x_1, \ldots, x_n\}$. The group $S_n$ acts to permute the basis, with the standard choice of roots and coroots. \end{ex}

\begin{ex} The \emph{affine permutation realization} of type $\tilde{A}_{n-1}$ is the free $\ringone$-module with basis $\{x_1, \ldots, x_n, \hbar\}$, with action and roots as in the introduction. Let $\epsilon_i \in V^*$ be defined by $\epsilon_i(x_j) = \delta_{ij}$ and $\epsilon_i(\hbar) = 0$. Then the coroots are $\al_i = \epsilon_i - \epsilon_{i+1}$ for $1 \le i \le n-1$, and $\al_0 = \epsilon_n - \epsilon_1$. \end{ex}

If $\al_s \ne 0$ for all $s \in S$, then an element $\hbar \in V$ is $W$-invariant if and only if $\al_s^\vee(\hbar) = 0$ for all $s \in S$. Note that $\al_s \ne 0$ if $\ringone$ does not have characteristic $2$, since $\al_s^\vee(\al_s) = 2 \ne 0$.

\begin{defn}
Let $\ringtwo$ be a subring of $\ringone$, and let $V$ be a realization over $\ringone$. We say that $W$ \emph{preserves a (full) $\ringtwo$-lattice} if we can find a basis of $V$ over $\ringone$ such that its $\ringtwo$-span $\La$ is preserved by the action of $W$. We do not assume that the roots live in $\La$. \end{defn}

\begin{defn}
Continue the setup of the previous definition. Suppose that $\ringone$ is a subring of $\R$, and let $\ringtwo_+ := \ringtwo \cap \R_{\ge 0}$ and $\ringtwo_- := \ringtwo \cap \R_{\le 0}$. An \emph{orthant} in a free $\ringtwo$-module $\La$ is the $\ringtwo_+$-span of a $\ringtwo$-basis of $\La$. We say that $W$ \emph{preserves a
$\ringtwo$-orthant} if $W$ preserves a $\ringtwo$-lattice $\La$, and preserves an orthant inside it. \end{defn}

% In other words, one can find a basis of $\La$ over $\ringtwo$ such that $W$ also preserves the $\ringtwo_+$-span of this basis. Equivalently, $W$ preserves the $\ringtwo_-$-span of this basis.

\begin{defn} Continue the setup of the previous definition, and let $W$ preserve a $\ringtwo$-lattice $\La$. An \emph{almost-orthant} in $\La$ is the set of elements of the form
\begin{equation} \label{almostorthant} \sum_i a_i x_i + \sum_j c_j \hbar_j, \qquad a_i \in \ringtwo_+, c_j \in \ringtwo, \end{equation}
where $\{x_i\} \cup \{\hbar_j\}$ is a $\ringtwo$-basis of $\La$, and $\hbar_j$ is $W$-invariant for all $j$. We typically write $X = \{x_i\}$ and $H = \{\hbar_j\}$ for the two parts of this basis, and call the set of vectors \eqref{almostorthant} the \emph{almost-orthant (for $\ringtwo$) spanned by $(X,H)$}. We say that $W$ \emph{preserves a $\ringtwo$-almost-orthant} if $W$ preserves a $\ringtwo$-lattice $\La$, and preserves an almost-orthant inside it. \end{defn}

%%%%%%%%%%%%%%%%%%%%%%
\subsection{Simple reflections preserving an almost-orthant}
%%%%%%%%%%%%%%%%%%%%%%

% I'm assuming that roots and coroots determine this reflection representation, but I'm not assuming any roots are in the orthant, or that they're even integral. What I mean is that, for any simple reflection, there is some element $\al_s$ in the rational vector space spanned by the lattice, and some element $\al_s^\vee$ in the dual space, such that
% \[ \al_s^\vee(\al_s) = 2, \qquad s(x) = x - \al_s^\vee(x) \cdot \al_s. \]
% I'm assuming roots and coroots form a root datum.

\begin{thm} \label{thm1} Fix a realization $V$ of $(W,S)$ over $\ringone \subset \R$, and suppose that $W$ preserves the almost-orthant for $\ringtwo$ spanned by $(X,H)$. Then for any
simple reflection $s \in S$, either $s$ acts trivially, or there exists $x \ne y \in X$, $b \in \ringtwo_+$, and $\hbar \in \Span_{\ringtwo} H$, such that $s(x) =
by + \hbar$, and $s(z) = z$ for all $z \in X \setminus \{x,y\}$. The element $b$ is invertible, with inverse in $\ringtwo_+$. The quadruple $(x,y,b,\hbar)$ is unique up to the symmetry $(x,y,b,\hbar) \leftrightarrow (y,x,b^{-1},-b^{-1} \hbar)$. \end{thm}

If $s$ acts nontrivially as above, we say that $s$ acts as a \emph{rescaled shifted transposition} which transposes $x$ and $y$.

\begin{proof}
If $W$ preserves the almost-orthant spanned by $(X,H)$ then it also preserves the almost-orthant spanned by $(-X,H)$, where $-X = \{-z \mid z \in X\}$. Note that $s(x) = by + \hbar$ if and only if $s(-x) = b(-y) - \hbar$. So $s$ has the desired form on the span of $(X,H)$ if and only if it has the desired form on $(-X,H)$. 

If $s$ does not act trivially, then choose some $x \in X$ such that $s(x) \ne x$. If $\al_s^\vee(x) < 0$, then replace $x$ with $-x$ and $X$ with $-X$. Thus we can assume without loss of generality that $\la := \al_s^\vee(x)$ is non-negative. Since $s(x) \ne x$, we have $\la > 0$.

By definition of the action we have $s(x) = x - \la \al_s$. Since this lies within the almost-orthant, we have
\begin{equation} x - \la \al_s = ax + \sum_i b_i y_i + \hbar, \end{equation}
for some $\hbar \in \Span_{\ringtwo} H$, $a \in \ringtwo_+$, and $b_i \in \ringtwo_+$, as $y_i$ ranges over $X \setminus \{x\}$. Thus
\begin{equation} \label{eq:al} \la \al_s = (1-a) x - \sum b_i y_i - \hbar. \end{equation}
Set $\mu_i := \al_s^\vee(y_i)$. Since $\al_s^\vee(\al_s) = 2$ and $\al_s^\vee(\hbar) = 0$, we have $2 = 1 - a - \sum \frac{b_i \mu_i}{\la}$, or 
\begin{equation} \label{eq:1ab} 1 + a = - \sum \frac{b_i \mu_i}{\la}. \end{equation}
(We can do this calculation in $\R$, where $\la$ is invertible.)
In particular, since $a \ge 0$ and $b_i \ge 0$ and $\la > 0$, there must be at least one $i$ such that $b_i \ne 0$ and $\mu_i < 0$.
Fix such an $i$, and consider $\la s(y_i)$:
\begin{equation} \label{syipos} \la s(y_i) = \la (y_i - \mu_i \al_s) = \mu_i(a-1) x + (\la + \mu_i b_i) y_i + \sum_{j \ne i} \mu_i b_j y_j + \mu_i \hbar . \end{equation}
Ignoring the term $\mu_i \hbar$, all coefficients on the right-hand side must live in $\ringtwo_+$.

First we claim that $b_j = 0$ for $j \ne i$. By \eqref{syipos} we have $\mu_i b_j \ge 0$. But $\mu_i < 0$ and $b_j \ge 0$ so $\mu_i b_j \le 0$.

Rewriting, we set $y = y_i$ and $b = b_i$ and $\mu = \mu_i$. Then
\begin{equation} 1 + a = - \frac{b \mu}{\la}, \qquad \la s(y) = \mu(a-1) x + (\la + \mu b) y + \mu \hbar. \end{equation}

By the first equation, $\la + \mu b = - \la a \le 0$. However, $\la + \mu b \ge 0$ because it appears as a coefficient in $\la s(y)$. Thus $\la + \mu b = 0$ and $a = 0$.

We conclude that $s(x) = by + \hbar$, where $b = \frac{-\la}{\mu} > 0$.

Now $\la \al_s = x - by - \hbar$ by \eqref{eq:al}. If $z \in X$, $z \ne x,y$ is any other basis element, let $\nu = \al_s^\vee(z)$. Then
\begin{equation} \la s(z) = \la z - \nu x + \nu b y + \nu \hbar. \end{equation}
The coefficient of $x$ is $-\nu \ge 0$, and the coefficient of $y$ is $\nu b \ge 0$. But $b > 0$, so the only possibility is $\nu = 0$. Hence $s(z) = z$.

Let $b^{-1}$ denote the coefficient of $x$ in $s(y)$, which must be in $\ringtwo_+$ by hypothesis. Since $s$ is an involution, $x = b s(y) + \hbar$, and thus $b \cdot b^{-1} = 1$. 
We have
\begin{equation} s(x) = by + \hbar \qquad \iff \qquad s(y) = b^{-1}(x - \hbar). \end{equation}
This proves that $b$ is invertible with inverse in $\ringtwo_+$, and that the symmetry in the statement does preserve valid quadruples. Since $s(z) = z$ for all $z \in X \setminus \{x,y\}$, the quadruple $(x,y,b,\hbar)$ is unique up to this symmetry.
\end{proof}

\begin{cor} In the setup of Theorem \ref{thm1}, if $\ringtwo = \Z$, then any $s \in S$ either acts trivially, or acts to permute $X$ by a single transposition modulo the span of $H$. \end{cor}

\begin{proof} The only invertible element of $\ringtwo^+$ is $1$, so the scalar $b$ in Theorem \ref{thm1} satisfies $b=1$. \end{proof}
	
Here is one of the results promised in the introduction.
	
\begin{cor} Suppose $W$ is a Coxeter group with a realization $V$ over $\ringone \subset \R$ which preserves a lattice $\La$ over $\ringtwo = \Z$, and preserves an orthant (not just an almost-orthant) inside.   
Then $W$ permutes a basis of $\La$, and the simple reflections act either trivially or as transpositions. \end{cor}

\begin{proof} We are in the setup of Theorem \ref{thm1} but with $H = \emptyset$. By the previous corollary, each simple reflection acts either trivially or by a transposition. \end{proof}

%%%%%%%%%%%%%%%%%%%%%%
\subsection{Coxeter groups faithfully preserving an orthant}
%%%%%%%%%%%%%%%%%%%%%%

We next prove that when $W$ acts faithfully and preserves an almost orthant over $\Z$, then $W$ is a product of finite and affine Weyl groups in type $A$. The assumption of
faithfulness is essential. We have already illustrated some non-faithful realizations in Remark \ref{rmk:nonfaithful}, and the proof below effectively points out a number of other non-faithful realizations. We work over a general subring $\ringtwo$ for as long as possible, and specialize to $\ringtwo = \Z$ when it become essential.

\begin{defn} \label{thegraph} Fix a realization $V$ of $(W,S)$ over $\ringone \subset \R$, and suppose that $W$ preserves the almost-orthant for $\ringtwo$ spanned by $(X,H)$. For each $s \in S$ which acts nontrivially on $V$, let $X_s = \{x,y\}$ be the unique pair of elements in $X$ such that $s(x) \ne x$ and $s(y) \ne y$, as guaranteed by Theorem \ref{thm1}. If $s \in S$ acts trivially, set $X_s = \emptyset$. The \emph{transposition graph} $\Gamma$ (corresponding to this setup) has vertex set $X$. For each $s \in S$ which acts nontrivally on $V$ there is an unoriented edge $e_s$ connecting the two elements of $X_s$. \end{defn}
	
In a Coxeter graph, the simple reflections correspond to vertices, while in a transposition graph they correspond to edges. To avoid confusion we use non-Coxeter terminology when discussing components of transposition graphs. We talk about the graph matching a type $A_n$ Dynkin diagram (for $n \ge 1$) as a \emph{line}, and talk about $\tilde{A}_n$ (for $n \ge 2$) as a \emph{circle}. We refer to two vertices connected by a double edge but not connected to any other vertices (i.e. $\tilde{A}_1$) as a \emph{lone double edge}.

We maintain the setup of Definition \ref{thegraph} throughout this section. The goal is to prove that if $W$ acts faithfully then $\Gamma$ has a very restricted form. When some $m_{st} = \infty$, we need to assume $\ringtwo = \Z$ to get the strongest restrictions.

First we argue that double edges correspond to reflections with $m_{st} = \infty$.

\begin{lem} Suppose that $X_s = X_t$. Then either $st$ acts trivially on $V$, or $st$ has infinite order on $V$. If $\ringtwo = \Z$ and $H = \emptyset$ then $st$ acts trivially. \end{lem}

\begin{proof} If both $s$ and $t$ act trivially the statement is obvious. Otherwise both act nontrivially. Set $X_s = \{x,y\} = X_t$. Then there exist $b_s, b_t \in
\ringtwo_+^\times$ and $\hbar_s, \hbar_t \in \Span_{\ringtwo} H$ such that $s(x) = b_s y + \hbar_s$ and $t(x) = b_t y + \hbar_t$. Then
\begin{equation}\label{eq:staction} st(x) = b_t b_s^{-1}(x) - b_t b_s^{-1} \hbar_s + \hbar_t.\end{equation} Thus the coefficient $x$ in $(st)^m (x)$ is $(b_t b_s^{-1})^m$. Since the only root of unity in $\ringtwo_+ \subset \R_+$ is $1$, if $st$
has finite order then $b_s = b_t$. In this case $st(x) = x + (\hbar_t - \hbar_s)$ is a translation. Nonzero translations (in a free module over a characteristic zero domain) have
infinite order. So if $st$ has finite order then $\hbar_t - \hbar_s = 0$, in which case $st$ acts trivially.

In the case when $\ringtwo = \Z$ and $H = \emptyset$, then $b_s = b_t = 1$ (as the only unit in $\ringtwo_+$) and $\hbar_s = \hbar_t = 0$. \end{proof}

Now we examine two edges meeting at a vertex.

\begin{lem} Suppose that $X_s \cap X_t$ has size $1$. Then the action of $st$ has order $3$, and $sts$ acts as a rescaled shifted transposition whose two non-fixed basis elements are the symmetric difference of $X_s$ and $X_t$. \end{lem}

\begin{proof} Let $X_s = \{x,y\}$ and $X_t = \{x,z\}$. Clearly $st$ does not act trivially. Let $b_s, b_t, \hbar_s, \hbar_t$ be such that $s(x) = b_s y + \hbar_s$ and $t(x) = b_t z + \hbar_t$. Note that $t$ fixes $s(x)$. Thus $sts(x) = x$. We compute that
\begin{equation} sts(z) = st(z) = s(b_t^{-1} (x - \hbar_t)) = b_t^{-1} (b_s y + \hbar_s - \hbar_t), \end{equation}
\begin{equation} tst(z) = ts(b_t^{-1} (x - \hbar_t)) = b_t^{-1} t(b_s y + \hbar_s - \hbar_t) = b_t^{-1}(b_s y + \hbar_s - \hbar_t). \end{equation}
In particular, $sts = tst$ on $z$ (and similarly on $y$), and acts as a rescaled shifted tranposition which transposes $y$ and $z$. \end{proof}

The next three lemmas argue that valence $\ge 3$ vertices are forbidden for faithful (integral) realizations. 

\begin{lem} Let $s, t, u$ be three simple reflections and $x \in X$ for which $X_s \cap X_t = X_s \cap X_u = X_t \cap X_u = \{x\}$. Then the action of $W$ is not faithful. We need not assume $\ringtwo = \Z$. \end{lem}

\begin{proof} Suppose that $X_s = \{x,y\}$ and $X_t = \{x,z\}$ and $X_u = \{x,w\}$.  If the action of $W$ were faithful, then $m_{st} = m_{su} = m_{tu} = 3$ by the previous lemma. Also, $sts$ is a rescaled shifted transposition on the set $\{y,z\}$, and thus commutes with $u$, which is a rescaled shifted transposition on the disjoint set $\{x,w\}$. So $(stsu)^2$ acts trivially. However, the element $(stsu)^2$ is not the identity in $W$. \end{proof}
	
\begin{remark} The quotient of $\tilde{A}_2$ by $(stsu)^2$ is isomorphic to $S_4$, just as in Remark \ref{rmk:nonfaithful}. \end{remark}

To rule out other ways of achieving valence $\ge 3$, we need to assume integrality. We begin by ruling out a triple edge.

\begin{lem} Suppose that $X_s = X_t = X_u \ne \emptyset$ for distinct simple reflections $s, t, u \in S$. If $\ringtwo = \Z$ then $W$ does not act faithfully. \end{lem}

\begin{proof} Let $X_s = X_t = X_u = \{x,y\}$. If the action of $W$ is faithful, then $m_{st} = m_{su} = m_{tu} = \infty$ by the previous lemma. As noted above, $b_s = b_t = b_u =
1$. Using \eqref{eq:staction} twice we have \begin{equation} sust(x) = x - 2 \hbar_s + \hbar_u + \hbar_t = stsu(x). \end{equation} It is easy to verify that $su$ and $st$ commute, whereas $sust \ne stsu$ in $W$, so $W$ does not act faithfully. \end{proof}
	
\begin{lem} Let $s, t, u$ be three distinct simple reflections with $X_s = X_t$ and where $X_s \cap X_u$ has size $1$. If $\ringtwo = \Z$ then $W$ does not act faithfully.
\end{lem}

\begin{proof} Let $X_s = X_t = \{x,y\}$ and $X_u = \{x,z\}$. If the action of $W$ is faithful then $m_{st} = \infty$ and $m_{su} = 3$ and $m_{tu} = 3$ by previous lemmas. As above, $b_s = b_t = b_u = 1$, so $s(x) = y + \hbar_s$ and $t(x) = y + \hbar_t$ and $u(x) = z + \hbar_u$. We have
\begin{equation} st(x) = x + \hbar_t - \hbar_s, \qquad st(y) = y + \hbar_s - \hbar_t, \qquad st(z) = z, \end{equation}
\begin{equation} ustu(x) = x, \qquad ustu(y) = y + \hbar_s - \hbar_t, \qquad ustu(z) = z + \hbar_t - \hbar_s. \end{equation}
These two operations commute, but $(ustust)^2$ is not the identity in $W$. \end{proof}

We can now conclude that if $W$ acts faithfully and $m_{st} < \infty$ for all $s, t \in S$ (but allowing $\ringtwo \ne \Z$), then the tranposition graph $\Gamma$ has no double edges or
vertices with valence $\ge 3$, so it is a disjoint union of lines and circles. Also, if $W$ acts faithfully and $\ringtwo = \Z$ (but
allowing $m_{st} = \infty$), then the tranposition graph $\Gamma$ has no vertices with valence $\ge 3$, so it is a disjoint union of lines and circles and lone double edges.

\begin{lem} Let $s, t \in S$. Suppose that $X_s \cap X_t = \emptyset$. Then the actions of $s$ and $t$ on $V$ commute. \end{lem}

\begin{proof} If either $s$ or $t$ acts trivially the statement is obvious. When both are nontrivial, the statement is still easy. \end{proof}

As a consequence, if $W$ acts faithfully then it is the direct product of the parabolic subgroups associated to (the edges in) each connected component of $\Gamma$. Moreover, we have computed $m_{st}$ for all $s, t \in S$ to be either $2$, $3$, or $\infty$ depending on the overlap of $X_s$ and $X_t$.

\begin{defn} If $C$ is a connected component of the transposition graph $\Gamma$, then $W_C$ is the parabolic subgroup of $W$ generated by the edges in $C$. \end{defn}

\begin{cor} \label{cor:itstypeA} Continue the setup of Definition \ref{thegraph}. If $W$ acts faithfully and either $\ringtwo = \Z$ or $m_{st} < \infty$ for all $s, t \in S$, then $W$ is a product of Coxeter systems in finite and affine type $A$. \end{cor}
	
\begin{proof} We need only prove the statement for each $W_C$. When $C$ is a line with $n$ vertices, $W_C \cong S_n$ (including the trivial case when $n=1$). When the component is a
lone double edge one obtains $\tilde{A}_1$, and when it is a circle with $n+1$ vertices one obtains $\tilde{A}_n$. These statements are easy to verify from the Lemmas above.
\end{proof}

Under the hypotheses of this corollary, we refer to the components which are lines as \emph{finite-type components}, and the components which are circles or lone double edges as \emph{affine-type components}.

%%%%%%%%%%%%%%%%%%%%%%
\subsection{Recovering the permutation action}
%%%%%%%%%%%%%%%%%%%%%%

Continue the setup of Definition \ref{thegraph}. It remains to argue that the realization $V$ can be reconstructed from the permutation actions of each component subgroup.

Suppose that $y \in X$ and $b \in \ringtwo_+^\times$ and $\hbar \in \Span_{\ringtwo} H$. Let $y' = by + \hbar$, and let $X'$ be
obtained from $X$ by replacing $y$ with $y'$. Then $(X,H)$ and $(X',H)$ span the same almost-orthant.

Let $C \subset \Gamma$ be a connected component of the transposition graph. Suppose that $C$ is a line with $n$ vertices.
Then one can label the vertices $\{x_1, x_2, \ldots, x_n\}$ and can find simple reflections $s_1, \ldots, s_{n-1}$ such that $X_{s_i} = \{x_i, x_{i+1}\}$. We know that $s_1(x_1) = b
x_2 + \hbar$ for some $b \in \ringtwo_+^\times$ and $\hbar \in \Span_{\ringtwo} H$. Replacing $x_2$ with $x_2' = s_1(x_1)$, we obtain a new basis $X'$ for the same almost-orthant,
and it is not hard to convince oneself that this new basis leads to the same transposition graph (replacing the vertex $x_2$ with $x_2'$). So by working with $X'$ instead of $X$, we
can assume that $s_1(x_1) = x_2$. Changing $x_3, x_4, \ldots, x_n$ in the same way, we can assume that $s_i(x_i) = x_{i+1}$ for all $1 \le i \le n-1$. We call this process
\emph{normalizing the basis} for the component $C$.

Normalizing the basis for one component will not affect the formulas for the action on any other component (in particular, we never change elements of the set $H$). For the next theorem we do not need the assumption that $\ringtwo=\Z$.

\begin{thm} If $W$ is a finite Coxeter group acting faithfully on a realization $V$ and preserving an orthant, then $W$ is a finite product of symmetric groups $W_C$ (including the
trivial group $S_1$), one for each component $C$ of the transposition graph. Moreover, $V \cong \bigoplus_C P_C \oplus Z$, where $P_C$ is the permutation representation of $W_C$ (inflated from this subgroup to $W$), and $Z$ is acted on trivially by $W$. \end{thm}

\begin{proof} This follows immediately from Corollary \ref{cor:itstypeA} and the discussion above, by normalizing the basis for each component, and letting $Z$ be the $\ringone$-span of $H$. \end{proof}

Now we discuss how to normalize affine-type components. Consider a component $C$ which is a lone double edge with vertices $\{x,y\}$ and edges $s$ and $t$. Replacing $y$ with $y'$ as
above, we can assume $s(x) = y$. Then $t(x) = by + \hbar_C$ for some $b \in \ringtwo_+^\times$ and $\hbar_C \in \Span_{\ringtwo} H$. If $\ringtwo = \Z$ then $b = 1$, and $t(y) = x -
\hbar_C$, exactly as in the affine permutation representation for $\tilde{A}_1$.

\begin{remark} If $\ringtwo \ne \Z$ then there are more options, e.g. $s(x) = y$ and $t(y) = qx$ where $q \in \R_+$ is neither $0$ nor $1$. This is a faithful realization of
$\tilde{A}_1$. In fact, it agrees with the linear part of $\tilde{R}_q$ as defined in the introduction, after specializing $q$ to a positive real number. One can make similar
realizations for $\tilde{A}_n$ with $n > 1$. \end{remark}

Now consider a component $C$ which is a circle with $n$ vertices for $n > 2$. After normalizing the basis for the first $n-1$ edges, we have $s_i(x_i) = x_{i+1}$ for all $1 \le i
\le n-1$, and $s_0(x_n) = b x_1 - \hbar$ for some $b$ and $\hbar$. As above, when $\ringtwo = \Z$ we know that $b=1$, and we recover the affine permutation representation. Note that the element $\hbar$ is the sum of the simple roots associated to the edges in $C$, so it can be picked out independent of the choice of basis; we may denote it $(\sum \al)_C$.

By normalizing each component, we have found a basis of the almost-orthant on which each $W_C$ acts as $P_C \oplus Z$, where $P_C$ is the (affine) permutation representation and
$Z$ is $W_C$-invariant. However, $V$ itself need not be the direct sum of the various $P_C$, since the elements $(\sum \al)_C$ associated to each component need not be linearly
independent.

\begin{thm} \label{thmwithmods} If $W$ is a Coxeter group acting faithfully on a realization $V$ and preserving an almost-orthant over $\ringtwo = \Z$, then $W$ is a finite product of symmetric groups
and affine Weyl groups $W_C$, one for each component of the transposition graph. Moreover, 
\[ V \cong (Z \oplus \bigoplus_C P_C)/K,\] where $P_C$ is the (affine) permutation
representation of $W_C$ (inflated from this subgroup to $W$), $Z$ is acted on trivially by $W$, and $K$ is the subspace of the $W$-invariants in $V$ described in the proof.
\end{thm}

\begin{proof} Suppose $W$ preserves the almost-orthant spanned by $(X,H)$. Without loss of generality we can assume that $X$ is normalized on each component as above. Let $H'$ be a
set in bijection with $H$, and let $V$ be the formal $\ringone$-span of $H'$. We now define a $\ringone$-linear map $\psi$ from $Z \oplus \bigoplus_C P_C$ to $V$. We send $Z$
isomorphically to the span of $H$. Let the (affine) permutation representation $P_C$ be defined as in \S\ref{subsec:setup}. Then we send the elements $\{x_i\}$ in $P_C$ to the
normalized basis for the component $C$, and send $\hbar_C \in P_C$ to $(\sum \al)_C$. The map $\psi$ is surjective, because it contains $X \cup H$ in its image. A straightforward,
row-reduction argument verifies that the kernel of $\psi$ (denoted $K$ above) is spanned by the elements $k_C$, one for each affine-type component $C$. Here $k_C = \hbar_C - v_C$,
where $\hbar_C \in P_C$ is the sum of the simple roots, and $v_C$ is the preimage in $Z$ of $(\sum \al)_C \in \Span H$. \end{proof}

\begin{remark} The map $\psi$ is a quotient morphism of realizations of $(W,S)$. It sends the roots of $Z \oplus \bigoplus_C P_C$ to the roots of $V$. Moreover, the
coroots kill the kernel $K$ so they descend to linear functions on the quotient, where they agree with the coroots in $V$. \end{remark}

\bibliographystyle{plain}
\bibliography{mastercopy}

\end{document}